\documentclass{amsart}

\usepackage{hyperref}
\usepackage{amsmath}
\usepackage{amsthm}

\newtheorem{theorem}{Theorem}[section]
\newtheorem{proposition}[theorem]{Proposition}
\newtheorem{corollary}[theorem]{Corollary}
\newtheorem{lemma}[theorem]{Lemma}

\theoremstyle{definition}
\newtheorem{definition}[theorem]{Definition}

\newtheorem{problem}[theorem]{Problem}

\theoremstyle{remark}

\numberwithin{equation}{section}

\newcommand{\R}{\mathbb{R}}

\newcommand{\C}{\mathbb{C}}

\begin{document}

\title[Orthogonal Polynomials and Differential Operators]{Real-Root Preserving
Differential Operator Representations of Orthogonal Polynomials}

\author[D.~A.~Cardon]{David~A.~Cardon}
\address{Department of Mathematics, Brigham Young University, Provo, UT 84602, USA}
\email{cardon@math.byu.edu}

\author[E.~Sorensen]{Evan~Sorensen}
\email{esorensencapps@gmail.com}

\author[J.~C.~White]{Jason~C.~White}
\email{white.jason.c@gmail.com}

\date{July 17, 2017}


\begin{abstract}
In this paper, we study linear transformations of the form $T[x^n]=P_n(x)$
where $\{P_n(x)\}$ is an orthogonal polynomial system. Of particular
interest is understanding when these operators preserve real-rootedness in
polynomials. It is known that when the $P_n(x)$ are the Hermite polynomials
or standard Laguerre polynomials, the transformation $T$ has this property.
It is also known that the transformation $T[x^n]=H_n^{\alpha}(x)$, where
$H_n^{\alpha}(x)$ is the $n$th generalized Hermite Polynomial with real
parameter $\alpha$, has the differential operator representation
$T[x^n]=e^{-\frac{\alpha}{2}D^2}x^n$. The main result of this paper is to
prove that a differential operator of the form $\sum_{k=0}^\infty
\frac{\gamma_k}{k!} D^k$ induces a system of monic orthogonal polynomials
if and only if $\sum_{k=0}^\infty \frac{\gamma_k}{k!}
D^k=\gamma_0e^{-\frac{ \alpha}{2}D^2-\beta D}$ where $\gamma_0,\alpha,\beta
\in \C$ and $\alpha,\gamma_0 \neq 0$. This operator will produce a shifted
set of generalized Hermite polynomials when $\alpha \in \R$. We also
express the transformation from the standard basis to the standard Laguerre
basis, $T[x^n]=L_n(x)$ as a differential operator of the form
$\sum_{k=0}^\infty \frac{p_k(x)}{k!} D^k$ where the $p_k$ are polynomials,
an identity that has not previously been shown.
\end{abstract}

\maketitle

\section{Introduction}

Let $T:\C[x]\rightarrow \C[x]$ be a linear transformation such that for every
real-rooted polynomial $p(x)$, the polynomial $T[p(x)]$ has real roots. Such
transformations are of particular interest when studying the zeros of entire
functions. In recent years, transformations involving orthogonal polynomials
have been considered. We are interested in transformations $T$ with the
real-root preserving property and the additional condition that for all $n,
T[x^n] = P_n(x)$, where the set of $P_n(x)$ form an orthogonal polynomial
system.

Many of the ideas involving orthogonal polynomials are motivated by reading
and understanding the concepts in
Chihara~\cite{Chihara_Orthogonal_Polynomials}. In his book, the following
definition is given.

\begin{definition}\cite[p.11]{Chihara_Orthogonal_Polynomials}
\label{Orthogonal_Polynomials} A sequence $\{P_n(x)\}_{n=0}^\infty$ is called
an \textit{orthogonal polynomial sequence} with respect to a moment
functional $\mathcal{L}$ provided for all nonnegative integers $m$ and $n$,
\begin{enumerate}
\item
$P_n(x)$ is a polynomial of degree $n$,
\item
$\mathcal{L}[P_m(x)P_n(x)]=0$ for  $m \neq n$, and
\item
$\mathcal{L}[P_n^2(x)] \neq 0$.
\end{enumerate}
\end{definition}
In most important cases, condition (3) can be replaced by
$\mathcal{L}[P_n^2(x)] > 0$, but that is not required in the most general
setting.

We will abbreviate orthogonal polynomial sequences by writing OPS in the
singular and plural senses. In the case that all polynomials in the set are
monic, we will call the set a monic OPS. One significant property of OPS is
that they follow a three-term recurrence relation. We summarize Theorems~4.1
and~4.4 of Chapter~1 in Chihara~\cite{Chihara_Orthogonal_Polynomials} as
follows. The referenced Theorem~4.4 is commonly known as Favard's Theorem.

\begin{theorem}
\cite[Thms. 4.1, 4.4, p. 18-22]{Chihara_Orthogonal_Polynomials}
\label{three-term_recurrence} $\{P_n(x)\}_{n=0}^\infty$ is a monic OPS if and
only if there exist sequences of constants $\{c_n\}$ and $\{\lambda_n\neq
0\}$ such that
\begin{equation} \label{monic_recurrence}
P_n(x)=(x-c_n)P_{n-1}(x)-\lambda_n P_{n-2}(x) \ \ \ \ n \geq 1
\end{equation}
\noindent where $P_0(x)=1$ and we define  $P_{-1}(x)=0$.
\end{theorem}

Note that the definition of orthogonal polynomials does not require the
system to be monic. In general, an OPS need not be monic, and the system
satisfies a recurrence of the form
\begin{equation} \label{Non-monic_recursion}
P_{n+1}(x)=(A_n x+B_n)P_n(x)-C_n P_{n-1}(x),
\end{equation}
with $A_n, C_n \neq 0$. Note that in this equation, the highest index is
shifted upward, as is standard in the literature for the non-monic case.

The above definition and equivalent recurrence relations give very general
definitions of orthogonal polynomials. It is often useful only to discuss
\textit{positive-definite} moment functionals, which ensure that the
orthogonal polynomials satisfy nice properties, such as the interlacing of
roots of successive polynomials. This also allows the moment functional to be
used as an inner product on the space of polynomials. Positive-definite
moment functionals will be discussed in more depth near the end of Section~2.

We will also make frequent reference to the differential operator $D$ in this
paper. $D$ represents differentiation with respect to $x$, so for a $k$-times
differentiable function  $f\colon\C \rightarrow \C$, we have  $D^k
f(x)=f^{(k)}(x)$.

A known result about linear transformations is included in
Piotrowski~\cite{Piotrowski_PhD}, which we will include here for convenience.

\begin{proposition} \label{Piotrowski_differential_operators}
\cite[Prop.~29, p.32]{Piotrowski_PhD}  Let $T\colon \mathbb{C}[x] \rightarrow \mathbb{C}[x]$ be a linear operator. Then, there exists a unique set of complex polynomials $\{p_k(x)\}_{k=0}^\infty$ such that
\[
T[f(x)]=\bigg(\sum_{k=0}^\infty \frac{p_k(x)}{k!} D^k \bigg)f(x)
\]
for all $f(x) \in \mathbb{C}[x]$.
\end{proposition}
We will note here that the expression given by Piotrowski does not include
the $k!$ expression. Since this only multiplies each polynomial by a scalar,
the statement is still true and will be useful in performing computations.

We will study the differential operator representations of the form above. We
hope that examining this for known transformations that preserve
real-rootedness will give insight into knowing about general transformations
that preserve real-rootedness and give an OPS.


\section{Differential Operators of the form \texorpdfstring{$\sum\limits_{k=0}^{\infty} \frac{\gamma_k}{k!}D^k$}{} }

The Hermite polynomials $\{H_n(x)\}$ play many important roles in physics,
probability, and numerical analysis, and they are discussed at length by
Piotrowski~\cite{Piotrowski_PhD}. They follow the recurrence relation
\begin{align*}
H_{n+1}(x) &= 2x H_{n}-2n H_{n-1}(x) \\
H_0(x) &= 1 \\
H_{-1}(x) &= 0 \\
\intertext{and can be expressed as}
H_n(x) &=2^n e^{-\frac{D^2}{4}}x^n.
\end{align*}
The Hermite polynomials can also be generalized by a real parameter $\alpha$
and satisfy the recurrence relation
\begin{align} \label{general_Hermite_recurrence}
H_n^\alpha(x) &= x H_{n-1}^\alpha(x)-\alpha(n-1)H_{n-2}^\alpha(x) \\ \nonumber
H_0^\alpha(x) &= 1 \\ \nonumber
H_{-1}^\alpha(x) &= 0.
\end{align}
Also, they can be related to the Hermite polynomials in the following way,
where $\alpha \neq 0$,
\[
H_n^{\alpha}(x)=\Big(\frac{\alpha}{2}\Big)^{n/2} H_n\Big(\frac{x}{\sqrt{2\alpha}}\Big).
\]
Furthermore, they can be represented by the differential operator
 \[
H_n^{\alpha}(x)=e^{-\frac{\alpha}{2} D^2}x^n.
\]

It is worth noting that this representation shows that for $\alpha \geq 0$,
the transformation is a real-root preserver. To see this, we introduce the
following class of functions.
\begin{definition}
The \textit{Laguerre-P\'olya class}, denoted by $\mathcal{LP}$, is the set
of functions obtained as uniform limits on compact sets of real polynomials
with real roots. They have the Weierstrass product representation
\[
cz^ne^{-\alpha z^2 + \beta z}\prod_{k=1}^{\infty}(1-\frac{z}{a_k})e^{\frac{z}{a_k}}\,,
\]
where $c, \, \alpha, \, \beta, \, a_k$ are real, $n$ is a non-negative
integer, $\alpha >0$, and $\sum_{k=1}^{\infty}|a_k|^{-2} < \infty$.
\end{definition}

From the Weierstrass product representation we see that $\phi(z) =
e^{-\frac{\alpha}{2}z^2} \in \mathcal{LP}$ for $\alpha > 0$. Then, as the
differential operators will only act on polynomials in this paper, the
following well-known theorem, originally proved by P\'olya, will suffice to
show that the transformation to the generalized Hermite polynomials (when
$\alpha > 0)$ is a real-root preserver.
\begin{theorem}\cite[Thm.~5.4.13, p.~157]{Rahman_and_Schmeisser} \label{LP-real-rootedness}
Assume
\[
\phi(z) = \sum_{k=0}^{\infty} a_kz^k \in \mathcal{LP}.
\]
Then, if $f(z)$ is a real polynomial with real-roots, $\phi(D)f(z)$ is also a
real polynomial with real roots.
\end{theorem}
For a more detailed presentation of the Laguerre-P\'olya class, as well as
the effect of various linear operators on the location of zeros, we highly
recommend chapters VIII and XI of Levin's book \cite{Levin}.

The operator $e^{-\frac{\alpha}{2}D^2}$ can be written as $\sum_{k=0}^\infty
\frac{(-\frac{\alpha}{2}D^2)^k}{k!}$, so in the representation of the linear
transformation given in Proposition~\ref{Piotrowski_differential_operators},
all of the $p_k(x)$ are constants. This raises the question of classifying
all such transformations to orthogonal polynomials that have the differential
operator representation in Proposition~\ref{Piotrowski_differential_operators} with the $p_k(x)$ constant. This brings us to the main result of
this paper.


\begin{theorem} \label{Main_Theorem}
Using the function $\phi(x)=\sum\limits_{k=0}^\infty \frac{\gamma_k}{k!}x^k$
as a differential operator, $\phi(D)x^n=P_n(x)$  gives an OPS if and only if
$\phi(x)=\gamma_0e^{-\frac{\alpha}{2}x^2-\beta x}$  with $\alpha, \beta,
\gamma_0 \in \C$ and $\alpha,\gamma_0 \neq 0$. Furthermore, the $P_n(x)$
satisfy the recurrence relation
    \begin{align*}
    P_n(x) & = (x-\beta)P_{n-1}(x)-\alpha(n-1)P_{n-2}(x) \\
        P_0(x) & = \gamma_0 \\
        P_{-1}(x) & = 0.
\end{align*}

\end{theorem}
For $\phi(x)$ as defined above, if we apply the differential operator
$\phi(D)$ to $x^n$ for all $n$, we can take
\begin{align*}
P_n(x)=\phi(D)[x^n]&=\Big(\sum\limits_{k=0}^\infty \frac{\gamma_k}{k!}D^k\Big)[x^n]
=\sum\limits_{k=0}^n \frac{\gamma_k}{k!} n(n-1) \ldots  (n-k+1)x^{n-k}
\\
&=\sum\limits_{k=0}^n \gamma_k \tbinom{n}{n-k} x^{n-k}
=\sum\limits_{k=0}^n \gamma_{n-k} \tbinom{n}{k} x^k \text{ for } n \geq 0.
\end{align*}
Notice that the leading term for each polynomial is
$\gamma_0$. From Definition~\ref{Orthogonal_Polynomials}, we know that we
must have $\gamma_0 \neq 0$ to ensure that each $P_n(x)$ has degree $n$.

We now prove two lemmas, which will allow us to prove Theorem~\ref{Main_Theorem}.

\begin{lemma} \label{constants_lemma}
Let $\phi(x)=\sum_{k=0}^\infty \frac{\gamma_k}{k!}x^k$ and
$\phi(D)x^n=P_n(x)$ for all $n$. The following are equivalent:
\begin{enumerate}
\item $\{P_n(x)\}_{n=0}^\infty$ is an OPS.
\item For $n \geq 1$, the set $\{P_n(x)\}_{n=0}^\infty$ follows the
    recurrence relation given in Equation~\eqref{monic_recurrence}, with
    $c_n$ and $\frac{\lambda_n}{n-1}\neq 0$ constant for all $n \geq 1$.
    Also, $P_0(x)=\gamma_0 \neq 0$, and again, we define $P_{-1}(x) = 0$.
\item For $n \geq 1$, the $\gamma_n$ defined above satisfy the recursion
    relation
\[
\gamma_n  = -b\gamma_{n-1}-a(n-1)\gamma_{n-2},
\]
where $b \in \C$ and $a \neq 0$ are the constants corresponding to $c_n$ and
$\frac{\lambda_n}{n-1},$ respectively. Also, we define $\gamma_{-1} = 0$, and
we have $\gamma_0 \neq 0$.
\end{enumerate}
\end{lemma}

A minor, but important note to make is that for $n = 1,
\frac{\lambda_n}{n-1}$ is undefined. However, in the recursion equation,
$\lambda_1$ is multiplied by $\gamma_{-1} = 0$, so we can choose $\lambda_1$
arbitrarily.

It is also rather important to note that in Theorem~\ref{Main_Theorem} and
Lemma~\ref{constants_lemma}, we do not assume that the OPS is monic. This
suggests that we need to prove that the above conditions are equivalent to
the corresponding system $\{P_n(x)\}_{n=0}^\infty$ satisfying the recursion
relation given in Equation~\eqref{Non-monic_recursion}, included below:
\[
P_{n+1}(x)=(A_nx+B_n)P_n(x)-C_nP_{n-1}(x).
\]
However, from our discussion immediately following the statement of
Theorem~\ref{Main_Theorem}, we showed that each of the $P_n(x)$ must have the
same leading term $\gamma_0$ in this case. This requires the $A_n$ in the
previous equation to be $1$ for all $n$. Therefore, it will suffice to show
that the above conditions are equivalent to the system satisfying the
recurrence of the form given in Equation~\eqref{monic_recurrence}, which we
include again here for convenience:
\[
P_n(x)=(x-c_n)P_{n-1}(x)-\lambda_nP_{n-2}(x),
\]
where we define $P_{-1}(x)=0$, but the only restriction on $P_0(x)=\gamma_0$
is that it is nonzero. Thus, in the case that $\gamma_0=1$, we will have a
monic OPS.

\begin{proof}[Proof of Lemma~\ref{constants_lemma}]
We will prove (1) $\Leftrightarrow$ (2), (1) $\Rightarrow$ (3), and (3)
$\Rightarrow$ (2). Note that (2) $\Rightarrow$ (1) is trivial by
Theorem~\ref{three-term_recurrence} and Equation~\eqref{Non-monic_recursion},
so we only need to prove that $c_n$ and $\frac{\lambda_n}{n-1}$ must be
constant given a monic OPS is induced by $\phi(D)$. From our remark above,
\[
P_n(x)=\sum\limits_{k=0}^n \gamma_{n-k} \binom{n}{k} x^k
\]
so to satisfy the three-term recurrence for an OPS, we must have that, for $n
\geq 2$,
\begin{equation} \label{recurrence:gamma_k}
\sum\limits_{k=0}^n \gamma_{n-k} \tbinom{n}{k} x^k=(x-c_n)\sum\limits_{k=0}^{n-1} \gamma_{n-1-k} \tbinom{n-1}{k} x^k-\lambda_n\sum\limits_{k=0}^{n-2} \gamma_{n-2-k} \tbinom{n-2}{k} x^k.
\end{equation}
In the case that $n = 1 $, defining $P_{-1} = 0$ gives us
\begin{equation} \label{recurrence:gamma_k:n=1:case}
\gamma_1 +\gamma_0 x = (x-c_1)\gamma_0.
\end{equation}
Comparing the coefficients of $x^{n-1}$ on each side of
Equation~\eqref{recurrence:gamma_k}, we get
\[
\gamma_1\tbinom{n}{n-1}x^{n-1}=\gamma_1\tbinom{n-1}{n-2}x^{n-1}-c_n \gamma_0\tbinom{n-1}{n-1}x^{n-1},
\]
so $n\gamma_1=(n-1)\gamma_1-c_n\gamma_0$ and $\gamma_1=-c_n\gamma_0.$ Now
note this calculation was independent of $n$, so $c_n$ must be constant. The
case of $n = 1$ gives the same result, simply by examining
Equation~\eqref{recurrence:gamma_k:n=1:case}. We will further denote $c_n$ as
$b$.

Now, we compare the coefficients of $x^{n-2}$ in Equation~\eqref{recurrence:gamma_k}.
\[
\gamma_2\tbinom{n}{n-2}x^{n-2}=\gamma_2\tbinom{n-1}{n-3}x^{n-2}-c_n\gamma_1\tbinom{n-1}{n-2}x^{n-2}-\lambda_n\gamma_0\tbinom{n-2}{n-2}x^{n-2},
\]
so
\[
\frac{n(n-1)}{2}\gamma_2=\frac{(n-1)(n-2)}{2}\gamma_2+b^2(n-1)-\lambda_n\gamma_0.
\]
Solving for $\gamma_2$ yields $\gamma_2=b^2-\frac{\lambda_n}{n-1}\gamma_0$.
Recall that we can choose $\lambda_1$ to be arbitrary. Again noting that this
calculation was independent of $n$, $\frac{\lambda_n}{n-1}$ must be constant,
which we will denote as $a$.

Now, to prove (1) $\Rightarrow$ (3), compare the constant terms from
Equation~\eqref{recurrence:gamma_k} to get the recurrence
\[
\gamma_n=-c_n\gamma_{n-1}-\lambda_n\gamma_{n-2}=-b\gamma_{n-1}-a(n-1)\gamma_{n-2}\ ,
\]
for $n \geq 2$, as desired. The case of $n = 1$ comes trivially from
Equation~\eqref{recurrence:gamma_k:n=1:case} by defining $\gamma_{-1} = 0$.
The condition that $\gamma_0 \neq 0$ has been discussed previously.

Now, to prove (3) $\Rightarrow$ (2), assume the three-term recurrence for
$\gamma_n$ holds for all $n \geq 1$ with $a \neq 0$ and $b$ as constants.
Also assume $\gamma_0 \neq 0$, and set $\gamma_{-1}=0$. As given above,
$P_n(x)=\sum\limits_{k=0}^n \gamma_{n-k} \binom{n}{k} x^k$. We can also write
this sum as $\sum\limits_{k=0}^{n} \gamma_{k} \binom{n}{n-k} x^{n-k}$. Then,
\begin{align}
P_n(x)& = \gamma_0x^n+\sum\limits_{k=1}^{n} (-b\gamma_{k-1}-a(k-1)\gamma_{k-2})\tbinom{n}{n-k}x^{n-k} \nonumber \\
& =\gamma_0x^n-b\sum\limits_{k=1}^{n} \gamma_{k-1} \tbinom{n}{n-k}x^{n-k}-a\sum\limits_{k=2}^n \gamma_{k-2}(k-1)\tbinom{n}{n-k}x^{n-k}
\label{Formula_for:P_n} \nonumber \\
& =\gamma_0x^n-b\sum\limits_{k=0}^{n-1} \gamma_k \tbinom{n}{n-k-1}x^{n-k-1}-a\sum\limits_{k=0}^{n-2} \gamma_k(k+1) \tbinom{n}{n-k-2}x^{n-k-2}.
\end{align}
Now note the following observations:
\begin{align*}
\tbinom{n}{n-k-1} &= \tfrac{n!}{(n-k-1)!(k+1)!}=\tfrac{n(n-1)!}{(n-k-1)!k!(k+1)} = \tbinom{n-1}{n-k-1}\frac{n}{k+1}  \\
\tbinom{n}{n-k-2}(k+1)& = \tfrac{n!(k+1)}{(n-k-2)!(k+2)!}=\tfrac{n(n-1)(n-2)!}{(n-k-2)!(k+2)k!} = \tbinom{n-2}{n-k-2}\frac{n(n-1)}{k+2}.
\end{align*}
Next, combining these observations with~\eqref{Formula_for:P_n}, we can
rewrite the expression for $P_n(x)$ as
\begin{align*}
&\gamma_0x^n-b\sum\limits_{k=0}^{n-1}\gamma_k\tbinom{n-1}{n-k-1}x^{n-k-1}+b\sum\limits_{k=0}^{n-1} \gamma_k\tbinom{n-1}{n-k-1}\big(1-\tfrac{n}{k+1}\big)x^{n-k-1} \\
&\qquad\qquad-a(n-1)\sum\limits_{k=0}^{n-2} \gamma_k \tbinom{n-2}{n-k-2}x^{n-k-2} +a\sum\limits_{k=0}^{n-2} \gamma_k \tbinom{n-2}{n-k-2}\big(n-1-\tfrac{n(n-1)}{k+2}\big)x^{n-k-2} \\
& =\gamma_0x^n-bP_{n-1}(x)-a(n-1)P_{n-2}(x)
+b\sum\limits_{k=0}^{n-1} \gamma_k\tbinom{n-1}{n-k-1}\big(1-\tfrac{n}{k+1}\big)x^{n-k-1} \\
 & \qquad\qquad +a\sum\limits_{k=0}^{n-2} \gamma_k \tbinom{n-2}{n-k-2}\big(n-1-\tfrac{n(n-1)}{k+2}\big)x^{n-k-2}.
\end{align*}
Since we are trying to prove that the sequence of polynomials satisfies the
three-term recurrence, it now suffices to show that
\[
x P_{n-1}(x) = \gamma_0x^n+b\sum\limits_{k=0}^{n-1} \gamma_k\tbinom{n-1}{n-k-1}\big(1-\tfrac{n}{k+1}\big)x^{n-k-1} +a\sum\limits_{k=0}^{n-2} \gamma_k \tbinom{n-2}{n-k-2}\big(n-1-\tfrac{n(n-1)}{k+2}\big)x^{n-k-2},
\]
which is equivalent to showing
\begin{align} \label{P_n-1}
P_{n-1}(x) = \gamma_0x^{n-1}& +b\sum\limits_{k=0}^{n-1} \gamma_k\tbinom{n-1}{n-k-1}\big(1-\tfrac{n}{k+1}\big)x^{n-k-2} \\ \nonumber
&+a\sum\limits_{k=0}^{n-2} \gamma_k \tbinom{n-2}{n-k-2}\big(n-1-\tfrac{n(n-1)}{k+2}\big)x^{n-k-3}.
 \end{align}
Now, consider the following calculation.
\begin{align*}
&\gamma_0x^{n-1}+b\sum\limits_{k=0}^{n-1} \gamma_k\tbinom{n-1}{n-k-1}\big(1-\tfrac{n}{k+1}\big)x^{n-k-2} = \gamma_0x^{n-1} + b\sum\limits_{k=0}^{n-2} \gamma_k\tbinom{n-1}{n-k-1}\big(1-\tfrac{n}{k+1}\big)x^{n-k-2} \\
&=\gamma_0x^{n-1}+b\gamma_0(1-n)x^{n-2}+b\sum\limits_{k=1}^{n-2} \gamma_k\tbinom{n-1}{n-k-1}\big(1-\tfrac{n}{k+1}\big)x^{n-k-2} \\
&=\gamma_0x^{n-1}+b\gamma_0(1-n)x^{n-2}+b\sum\limits_{k=2}^{n-1} \gamma_{k-1}\tbinom{n-1}{n-k}\big(1-\tfrac{n}{k}\big)x^{n-k-1},
\end{align*}
which after some manipulation, results in
\begin{align*}
&\gamma_0x^{n-1}+a \gamma_0(1-n)x^{n-2}-b\sum\limits_{k=2}^{n-1} \gamma_{k-1}\tbinom{n-1}{n-k-1}x^{n-k-1}\\
=\ &\gamma_0x^{n-1}+\gamma_1(n-1)x^{n-2}-b\sum\limits_{k=2}^{n-1} \gamma_{k-1}\tbinom{n-1}{n-k-1}x^{n-k-1},
\end{align*}
the last step coming from the recurrence relation for the $\gamma_n$.
A similar calculation shows that
\[
a\sum\limits_{k=0}^{n-2} \gamma_k \tbinom{n-2}{n-k-2}\big(n-1-\tfrac{n(n-1)}{k+2}\big)x^{n-k-3} =- a\sum\limits_{k=2}^{n-1} (k-1)\gamma_{k-2}\tbinom{n-1}{n-k-1}x^{n-k-1}.
\]
Putting these calculations together and again using the three-term recurrence
for $\gamma_n$, Equation~\eqref{P_n-1} holds. This completes the proof of
Lemma~\ref{constants_lemma}.
\end{proof}

\begin{lemma} \label{differential_equation}
Let $\phi(x) = \sum_{k=0}^\infty \frac{\gamma_k}{k!} x^k$. For every $a,b \in \C$, $\phi(x)$ satisfies the differential equation
\[
\phi''(x)+(ax+b)\phi'(x)+a\phi(x)=0
\]
if and only if for $n \geq 2$, the $\gamma_n$ satisfy the recurrence relation
\begin{equation} \label{diffeq_coeff_recursion}
\gamma_n = -b\gamma_{n-1}-a(n-1)\gamma_{n-2}.
\end{equation}
\end{lemma}
As a caution, we note that the conclusion of this lemma does not quite
satisfy condition (2) of Lemma~\ref{constants_lemma} since this result only
holds true for $n \geq 2$.

\begin{proof}
First assume the $\gamma_n$ satisfy the recursion
relation~\eqref{diffeq_coeff_recursion} for $n \geq 2$. Then,
\begin{align}\label{Differential_Equation_Recurrence}
\phi(x)&=\sum\limits_{k=0}^\infty \frac{\gamma_k}{k!}x^k = \gamma_0+\gamma_1x+\sum\limits_{k=2}^\infty \frac{\gamma_k}{k!}x^k \nonumber \\
&=\gamma_0+\gamma_1x+\sum\limits_{k=2}^\infty\frac{-b\gamma_{k-1}-a(k-1)\gamma_{k-2}}{k!}x^k \nonumber \\
&=\gamma_0+\gamma_1x-b\sum\limits_{k=2}^\infty \frac{\gamma_{k-1}}{k!}x^k-a\sum\limits_{k=2}^\infty\frac{(k-1)\gamma_{k-2}}{k!}x^k \nonumber \\
&=\gamma_0+\gamma_1x+bx-b\sum\limits_{k=1}^\infty \frac{\gamma_{k-1}}{k!}x^k-ax\sum\limits_{k=2}^\infty\frac{\gamma_{k-2}}{(k-1)!}x^{k- 1}+a\sum\limits_{k=2}^\infty\frac{\gamma_{k-2}}{k!}x^k.
\end{align}
By manipulating the series expression for $\phi(x)$ and shifting indices as
needed, we obtain the following:
\[
\Big(\sum\limits_{k=2}^\infty\frac{\gamma_{k-2}}{k!}x^k\Big)''=\Big(\sum\limits_{k=1}^\infty \frac{\gamma_{k-1}}{k!}x^k\Big)'=\Big(\sum\limits_{k=2}^\infty\frac{\gamma_{k-2}}{(k-1)!}x^{k-1}\Big)'=\phi(x).
\]
Now, combining these observations with
Equation~\eqref{Differential_Equation_Recurrence}, differentiating twice and
moving all terms to the left side, we obtain
\[
\phi''(x)+(ax+b)\phi'(x)+a\phi(x)=0.
\]
Next, assume that $\phi(x)$ satisfies the given differential equation. Then,
differentiating the expression for $\phi(x)$, we have
\[
\sum_{k=2}^\infty \frac{\gamma_k}{(k-2)!} x^{k-2} + (ax+b) \sum_{k=1}^\infty \frac{\gamma_k}{(k-1)!} x^{k-1} + a \sum_{k=0}^\infty \frac{\gamma_k}{k!} x^k=0
\]
Now, let $n \geq 2$ and compare the coefficients of $x^{n-2}$ in the above
expression. This yields
\[
\frac{\gamma_n}{(n-2)!} + a\frac{\gamma_{n-2}}{(n-3)!} + b\frac{\gamma_{n-1}}{(n-2)!} + a\frac{\gamma_{n-2}}{(n-2)!}=0.
\]
Multiplying by $(n-2)!$ gives us the recurrence formula
\[
\gamma_n=-b\gamma_{n-1}-a(n-1)\gamma_{n-2}
\]
for $n \geq 2$.
\end{proof}

With the above lemmas, we are ready to prove the main result.


\begin{proof}[Proof of Theorem \ref{Main_Theorem}]
It is a simple exercise to show that, for all $\gamma_0$,
$\gamma_0e^{-\frac{\alpha}{2}x^2-\beta x}$ is a solution of the differential
equation
\[
\phi''(x)+(\alpha x + \beta)\phi'(x)+ \alpha \phi(x)=0.
\]
By Lemma~\ref{differential_equation}, if we express $\phi(x)=\gamma_0
e^{-\frac{\alpha}{2}x^2-\beta x}$ as $\sum_{k=0}^\infty
\frac{\gamma_k}{k!}x^k$ (note that the use of $\gamma_0$ is consistent since
$\phi(0)=\gamma_0$ in both cases), we know that for $n \geq 2$, we get the
recurrence
\[
\gamma_n = -\beta \gamma_{n-1} - \alpha(n-1)\gamma_{n-2}.
\]
Note $\gamma_1 = \phi'(0) = -\beta \gamma_0$, so by defining $\gamma_{-1}$ to
be zero, the recurrence holds for all $n \geq 1$. Then, by
Lemma~\ref{constants_lemma}, when $\gamma_0, \alpha \neq 0, \phi(x) =
\gamma_0 e^{-\frac{\alpha}{2}x^2 -\beta x}$, and where we define
$P_n(x)=\phi(D)x^n$ for all $n$, the set of $P_n(x)$ form an OPS satisfying
the recurrence
\begin{align*}
P_n(x) &= (x-\beta)P_{n-1}(x) - \alpha(n-1)P_{n-2}(x)  \ \ \ \ n \geq 1\\
P_0(x) &= \gamma_0 \\
P_{-1}(x) &= 0.
\end{align*}
This proves one direction of Theorem~\ref{Main_Theorem}.

Now, if we assume that $\phi(D)x^n = \big(\sum_{k=0}^\infty
\frac{\gamma_k}{k!}D^k\big)x^n = P_n(x)$ for all $n$ and this forms an OPS,
we know from Lemma~\ref{constants_lemma} that for $n\geq 1$, the $\gamma_n$
must satisfy the recurrence
\[
\gamma_n=-b \gamma_{n-1} - a(n-1)\gamma_{n-2},
\]
where $a, b, \gamma_0 \in \C, a, \gamma_0 \neq 0$, and we define $\gamma_{-1}
= 0$.  Furthermore, by Lemma~\ref{differential_equation}, $\phi(x)$ must
satisfy the differential equation
\[
\phi''(x)+(a x + b)\phi'(x)+ a \phi(x)=0.
\]
We can also conclude that for this problem, since $\phi(x)=\sum_{k=0}^\infty
\frac{\gamma_k}{k!}x^k$, $\phi(0)=\gamma_0$. Also from the recursion relation
among the $\gamma_n$, we have that $\phi'(0)=\gamma_1=-a\gamma_0$. Given
these conditions, basic knowledge of differential equations tells us that the
solution $\phi(x) = \gamma_0 e^{-\frac{a}{2}x^2-b x}$  is unique. This
completes the proof of Theorem~\ref{Main_Theorem}.
\end{proof}

Note that the choice of $\gamma_0$ simply multiplies all elements of the OPS
by $\gamma_0$, which is the leading term of each polynomial in the system.
This gives us the following corollary.
\begin{corollary}
For the sum $\phi(x)=\sum\limits_{k=0}^\infty \frac{\gamma_k}{k!}x^k$, the
operator $\phi(D)x^n=P_n(x)$  gives a monic OPS if and only if
$\phi(x)=e^{-\frac{\alpha}{2}x^2-\beta x}$  where $\alpha \neq 0$.
\end{corollary}

In the beginning of the proof of Theorem~\ref{Main_Theorem}, we showed that
for $\gamma_0e^{-\frac{\alpha}{2}x^2-\beta x}=\sum_{k=0}^\infty
\frac{\gamma_k}{k!}x^k$, the $\gamma_n$ satisfy the recurrence relation
\[
\gamma_n = -\beta\gamma_{n-1}-\alpha(n-1)\gamma_{n-2}.
\]
By Lemma~\ref{constants_lemma}, this implies that for $P_n(x) =
\gamma_0e^{-\frac{\alpha}{2}D^2-\beta D}x^n,$ the set of $P_n(x)$ follow the
three-term recurrence relation
\[
P_n(x)=(x-\beta)P_{n-1}(x)-\alpha(n-1)P_{n-2}(x).
\]
Now recall, as in our discussion at the beginning of this section, that the
generalized Hermite polynomials $H_n^\alpha(x)=e^{-\frac{\alpha}{2}D^2}x^n$
for real $\alpha$ follow the three-term recurrence relation
\[
H_n^\alpha(x)=xH_{n-1}^\alpha(x)-\alpha(n-1)H_{n-2}^\alpha(x).
\]
From Chihara~\cite[p.108]{Chihara_Orthogonal_Polynomials}, we know that if
$Q_n(x)$ is an OPS with $c_n$ and $\lambda_n$ as the constants of the
three-term recurrence and we have
\[
R_n(x)=Q_n(x+s),
\]
then the $R_n$ satisfy the three-term recurrence
\[
R_n(x)=(x-(c_n-s))R_{n-1}(x)-\lambda_n R_{n-2}(x), \quad n \geq 1.
\]
Given $P_0=1$ and setting $P_{-1}(x)=0$, the three-term recurrence relation
uniquely determines the system, so we see that a shift in the $c_n$ gives a
shift in the OPS. This gives us the following observation.

\begin{lemma} \label{root_preserving_lemma}
Whenever $\alpha \in \R$,
\[
e^{-\frac{\alpha}{2}D^2-\beta D}x^n=H_n^\alpha(x-\beta).
\]
 Furthermore, if $\alpha >0 $ and $\beta$ is real, the linear transformation
 $T: \R[x] \rightarrow \R[x]$ defined by $T[x^n] = H_n^{\alpha}(x-\beta)$ is
 such that whenever $p(x) \in \R[x]$ is a polynomial with only real roots,
 $T[p(x)]$ also has only real roots.
\end{lemma}

Note that the condition that $\beta$ is real just shifts all roots by a real
number, which justifies the statement that the operator preserves
real-rootedness. The condition that $\alpha > 0$ comes from the discussion of
Theorem~\ref{LP-real-rootedness}.

A topic of particular interest in
Chihara~\cite{Chihara_Orthogonal_Polynomials} relates to the moment
functional $\mathcal{L}$ given in Definition~\ref{Orthogonal_Polynomials}. We
include the following definition.

\begin{definition} \cite[p.13]{Chihara_Orthogonal_Polynomials}
A moment functional $\mathcal{L}$ is called \textit{positive-definite} if
$\mathcal{L}[\pi(x)]>0$ for every polynomial $\pi(x)$ that is not identically
zero and is non-negative for all real $x$.
\end{definition}

This condition causes the zeros of each polynomial in the OPS to satisfy
certain properties. These include each polynomial having real roots and
interlacing of the roots of successive polynomials. A very useful thing to
note is that for an OPS $\{P_n(x)\}_{n=0}^\infty$ satisfying the recurrence
\[
P_n(x) =(x-c_n)P_{n-1}(x)-\lambda_nP_{n-2}(x),
\]
with $\lambda_n \neq 0$, the corresponding moment functional $\mathcal{L}$ is
positive-definite if and only if $c_n$ is real and $\lambda_n > 0$. This is
another piece of the well-known Favard's Theorem 4.4 in
Chihara~\cite{Chihara_Orthogonal_Polynomials}. We now have the following
theorem, which follows from Theorem~\ref{Main_Theorem}, Lemma~\ref{root_preserving_lemma} the above observations, and the recurrence
relation~\eqref{general_Hermite_recurrence} .
\begin{theorem}
$\phi(D)x^n = (\sum_{k=0}^{\infty} \frac{\gamma_k}{k!} D^k )x^n=P_n(x)$ gives
an OPS with a positive-definite moment functional $\mathcal{L}$ if and only
if $P_n(x)=\gamma_0H_n^{\alpha}(x-\beta)$ for all $n$ with $\alpha, \beta \in
\R$, $\alpha > 0$, and $\gamma_0 \neq 0$. Specifically,
$\gamma_0H_n^{\alpha}(x-\beta)=\gamma_0 e^{-\frac{\alpha}{2}D^2-\beta D}x^n$,
and the differential operator $\phi(D)$ preserves real-rootedness.
\end{theorem}

\section{Another Example of a real-root preserving Differential Operator}

The Laguerre polynomials are another type of OPS that depend on a real parameter $\alpha$. (Note that some authors only define these for $\alpha > -1$.) They have the following well-known closed form expression:
\begin{equation} \label{Laguerre_closed_form}
L_n^\alpha(x) = \sum_{r=0}^n \frac{(-1)^r}{r!} \binom{n+\alpha}{n-r} x^r.
\end{equation}
It was proved by Fisk~\cite{Fisk_Laguerre_Polynomials} that the transformation $T[x^n]=L_n(x)$, where the $L_n(x)$ are the standard Laguerre polynomials ($\alpha=0$), preserves real-rootedness. In this section, we construct the explicit differential operator representation of this transformation. As far as we know, this expression is new.

\begin{theorem} \label{Laguerre_differential_operator_theorem}
The transformation to the standard Laguerre polynomials ($\alpha=0 $) can be
expressed as a differential operator by
\begin{equation} \label{differential_operator_Laguerre}
L_n(x)=\Big(\sum_{k=0}^{\infty} \frac{p_k(x)}{k!}D^k\Big)[x^n],
\end{equation}
where $L_n$ is the $n^{th}$ Laguerre Polynomial and
\begin{equation} \label{expression_for:p_n(x)}
p_n(x)= \sum_{r=0}^n \sum_{l=0}^{r} \frac{  \binom{n}{r}  \binom{r}{l} (-1)^r  }{  l!  }  x^r \quad\text{ for all } n.
\end{equation}
\end{theorem}
By Proposition~\ref{Piotrowski_differential_operators}, a unique
representation of the form in Equation~\eqref{differential_operator_Laguerre}
exists where $p_k(x)$ is a polynomial for all $k$.
Piotrowski~\cite{Piotrowski_PhD} also shows in the proof of this proposition
that the $p_k(x)$ can be given recursively by
\begin{align*}
p_0(x) & = T[1]  \\
p_n(x) & = T[x^n]-\sum_{k=0}^{n-1} \frac{p_k(x)}{k!} D^k x^n,
\end{align*}
where $T$ represents the linear transformation from $x^n$ to $L_n(x)$. Noting
that $T[x^n]=L_n(x)$ has degree $n$, the above formula inductively shows that
$p_n(x)$ has degree at most $n$ for all $n$. Hence, we can write
\begin{equation} \label{p_n:in_terms_of_coefficients}
p_n(x)= \sum_{r=0}^n q_{n,r}~x^r \text{ for all } n \geq 0,
\end{equation}
where the $q_{n,r}$ are constants. With this notation in place, we are now
able to prove the following lemmas.


\begin{lemma} \label{representation_of_coefficients}
For all $n,r$ with $0 \leq r \leq n$, we have  $q_{n,r}=\binom{n}{r}a_r$
where $a_0=1$, and for $r \geq 1$, the following recurrence relation holds:
\[
a_r=\frac{(-1)^r}{r!}-\sum_{k=0}^{r-1} \binom{r}{k} a_k.
\]
\end{lemma}

\begin{proof}
Setting $\alpha = 0$, we get from Equation~\eqref{Laguerre_closed_form} that
\[
L_n(x)=\sum_{r=0}^n \binom{n}{r} \frac{(-1)^r}{r!} x^r .
\]
We can then combine this with Equation~\eqref{differential_operator_Laguerre}
to obtain
\begin{gather*}
\sum_{r=0}^n \binom{n}{r} \frac{(-1)^r}{r!} x^r=(\sum_{k=0}^{\infty} \frac{p_k(x)}{k!}D^k)[x^n]=\sum_{k=0}^{n} p_k(x) \frac{n(n-1)\ldots(n-k+1)}{k!}x^{n-k} \\
=\sum_{k=0}^{n} p_k(x) \binom{n}{n-k} x^{n-k}=\sum_{k=0}^{n} p_{n-k}(x) \binom{n}{k} x^{k}.
\end{gather*}
We will now compare coefficients of each power
of $x$ in the equation
\[
\sum_{r=0}^n \binom{n}{r} \frac{(-1)^r}{r!} x^r = \sum_{k=0}^{n} p_{n-k}(x) \binom{n}{k} x^{k}.
\]
Comparing constant terms yields $1=p_{n,0}.$ This must hold true for all $n$.
Now comparing coefficients of $x$, we obtain
\[
\binom{n}{1} \frac{(-1)^1}{1!} x = q_{n-1,0} \binom{n}{n-1} x + q_{n,1} \binom{n}{n} x,
\]
which yields
\[
q_{n,1}=-\binom{n}{1} - q_{n-1,0}\binom{n}{1}=-2 \binom{n}{1}
\]
since $p_{n,0}=1$ for all $n$. This proves the lemma for $n\leq 1$. In
general, we see that
\[
\binom{n}{r} \frac{(-1)^r}{r!} x^r = \sum_{k=0}^{r} q_{n-k,r-k} \binom{n}{k} x^r,
\]
which gives the equation
\begin{equation} \label{recursive_for:p_n,r}
q_{n,r} = \binom{n}{r}\frac{(-1)^r}{r!}-\sum_{k=1}^{r} q_{n-k,r-k} \binom{n}{k}.
\end{equation}
Now, assume inductively that $n>1$ and that for all $m<n$,
$p_{m,r}=\binom{m}{r} a_r$, where $a_r$ does not depend on $n$. Note that a
simple manipulation of binomial coefficients gives $\binom{n-k}{r-k}
\binom{n}{k} = \binom{n}{r} \binom{r}{k}$. Then, from
Equation~\eqref{recursive_for:p_n,r},
\begin{gather*}
q_{n,r}=\binom{n}{r} \frac{(-1)^r}{r!}-\sum_{k=1}^{r} \binom{n-k}{r-k} \binom{n}{k} a_{r-k}=\binom{n}{r} \frac{(-1)^r}{r!}-\sum_{k=1}^{r} \binom{n}{r} \binom{r}{k} a_{r-k} \\
= \binom{n}{r} \Big(\frac{(-1)^r}{r!}-\sum_{k=1}^{r} \binom{r}{k} a_{r-k}\Big)=
\binom{n}{r} \Big(\frac{(-1)^r}{r!}-\sum_{k=0}^{r-1} \binom{r}{r-k} a_{k}\Big).
\end{gather*}
Simply noting $\binom{r}{r-k} = \binom{r}{k}$ proves the lemma.
\end{proof}


\begin{lemma} \label{Weird_way_to_write:(-1)^r/r!}
The following identity holds for all $r\geq 0$:
\[
\sum_{k=0}^{r} \sum_{l=0}^{k} \frac{ \binom{r}{k} \binom{k}{l} (-1)^k }{l!}
=\frac{(-1)^r}{r!}.
\]
\end{lemma}

\begin{proof}
This identity is proved by changing the order of summation. By changing the
order of $k$ and $l$ in the sum on the left, we obtain
\[
\sum_{k=0}^{r} \sum_{l=0}^{k} \frac{ \binom{r}{k} \binom{k}{l} (-1)^k }{l!}
=
\sum_{l=0}^{r} \sum_{k=l}^{r} \frac{ \binom{r}{k} \binom{k}{l} (-1)^k}{l!}
=
\frac{ (-1)^r }{ r! } + \sum_{l=0}^{r-1} \sum_{k=l}^{r} \frac{ \binom{r}{k} \binom{k}{l} (-1)^k}{l!}.
\]
By simple comparison from the definition of binomial coefficients, we note
that $\binom{r}{k} \binom{k}{l} = \binom{r}{l} \binom{r-k}{k-1}$. The above
sum then becomes
\[
\frac{ (-1)^r }{ r! }+\sum_{l=0}^{r-1} \frac{ \binom{r}{l} }{l!} \sum_{k=l}^{r} \binom{r-k}{k-l} (-1)^k,
\]
which after a change of variable in the second sum is
\[
\frac{(-1)^r}{r!}+\sum_{l=0}^{r-1} \frac{ \binom{r}{l} }{l!} \sum_{k=0}^{r-l} \binom{r-l}{k} (-1)^{k+l}
=\frac{(-1)^r}{r!}+\sum_{l=0}^{r-1} \frac{ \binom{r}{l} }{l!}(-1)^{l}(1-1)^{r-l} = \frac{(-1)^r}{r!}.
\]
\end{proof}


\begin{lemma} \label{closed_form_satisfies_recursion}
The closed-form expression
\begin{equation} \label{closed-form-solution-a_r}
a_r=(-1)^r \sum_{l=0}^r \frac{ \binom{r}{l} }{l!}
\end{equation}
is the unique solution to the recursion formula
\begin{equation} \label{recursive_formula:a_r}
a_r=\frac{(-1)^r}{r!}-\sum_{k=0}^{r-1} \binom{r}{k} a_k
\end{equation}
such that $a_0=1$.
\end{lemma}

\begin{proof}
We will assume that~\eqref{closed-form-solution-a_r} holds for all $r\geq 0$
and then prove that this satisfies equation~\eqref{recursive_formula:a_r}.
Note that we can rewrite~\eqref{recursive_formula:a_r} as
\[
a_r=\frac{(-1)^r}{r!}-\sum_{k=0}^{r} \binom{r}{k} a_k +a_r,
\]
which is equivalent to
\[
\frac{(-1)^r}{r!}=\sum_{k=0}^r \binom{r}{k} a_k.
\]
Substituting Equation~\eqref{closed-form-solution-a_r} in for the $a_k$ on
the right hand side and directly applying Lemma~\ref{Weird_way_to_write:(-1)^r/r!} proves the lemma.
\end{proof}
Lemmas~\ref{closed_form_satisfies_recursion}
and~\ref{representation_of_coefficients} combined with
Equation~\eqref{p_n:in_terms_of_coefficients} prove
Theorem~\ref{Laguerre_differential_operator_theorem}.


\section{Open Problems and Further Research}
In this paper, we described the differential operator representation of two
types of real-root preserving linear transformations. In Borcea and
Br\"and\'en~\cite{Borcea_Branden}, a classification for all linear operators
that preserve real-rootedness is given. A natural problem following these
results is to classify all linear operators that preserve real-rootedness and
are of the form $T[x^n]=P_n(x)$ where $\{P_n(x)\}$ is an OPS. In general, we
do not expect an OPS to satisfy easily accessible formulas as is the case
with a classical OPS. However, we do know that every OPS satisfies a
three-term recurrence relation
\begin{equation}
P_{n+1}(x)=(A_n x+B_n)P_n(x)-C_n P_{n-1}(x),
\end{equation}
with $A_n, C_n \neq 0$. So far, our attempts on this more general problem
have not been successful because of the difficulty working with arbitrary
sequences of constants in the recurrence relation.

\begin{problem} \label{problem 1}
Classify all real-root preserving transformations $T$ such that
$T[x^n]=P_n(x)$ for all $n$ where $\{P_n(x)\}_{n=0}^\infty$ is an OPS.
\end{problem}

At the beginning of Section 2, we made a few short comments about the
standard Hermite polynomials. We noted that they have the differential
operator expression $H_n(x)=2^n e^{-\frac{D^2}{4}}x^n$. Note that this is not
of the form $\gamma_0 e^{-\frac{\alpha}{2}D^2-\beta D}x^n$ because of the
extra $2^n$ scalar. Thus, from Theorem~\ref{Main_Theorem}, we know that $2^n
e^{-\frac{D^2}{4}}$ is not of the form $\sum_{k=0}^\infty
\frac{\gamma_k}{k!}D^k$. We can see from
Definition~\ref{Orthogonal_Polynomials} that multiplying the polynomials in
an OPS by nonzero constants does not change the orthogonality of the system.
Thus, the operator $e^{-\frac{D^2}{4}}x^n$ gives an OPS, and multiplying by
the $2^n$ term scales each of the polynomials in the set. It is well known
that the transformation $T[x^n]=H_n(x)$ preserves real-rootedness as a
consequence of the result quoted in Theorem~\ref{LP-real-rootedness}. This
suggests the following problem.

\begin{problem}
Express $T[x^n]=H_n(x)$ as $T[x^n] = \big(\sum_{k=0}^\infty \frac{p_k(x)}{k!}
D^k \big)x^n$ in closed form.
\end{problem}
It appears that the above problem is not too difficult because this
transformation is only the rescaling of a known differential operator.

In proving Theorem~\ref{Laguerre_differential_operator_theorem}, we also
attempted to find the differential operator representation for $T[x^n] =
L_n^\alpha(x)$ where the $L_n^\alpha(x)$ are the generalized Laguerre
polynomials with $\alpha \in \R$ arbitrary. However, the extra $\alpha$ term
in the expression for these polynomials,
\[
L_n^\alpha(x) = \sum_{r=0}^n \frac{(-1)^r}{r!} \binom{n+\alpha}{n-r} x^r,
\]
made it so that the binomial relationships involved were much more
complicated. It is possible that the following problem could be solved by
applying similar methods to those in this paper and developing some new
clever ideas.

\begin{problem}
Find the differential operator representation for the transformation
$T[x^n]=L_n^\alpha(x)$ where $\alpha \in \R$ is arbitrary.
\end{problem}

Another interesting problem deals with the classification of real-root
preserving operators given in Borcea and Br\"and\'en~\cite{Borcea_Branden}.
Two characterizations of these operators are given, but we are also
interested in describing an arbitrary linear real-root preserver as a
differential operator in the form $T[f(x)]=\sum_{k=0}^\infty
\frac{p_k(x)}{k!} f^{(k)}(x).$

\begin{problem} \label{differential_operator_problem}
Given an arbitrary real-root preserving linear transformation producing an
OPS, describe its representation as a differential operator in closed form
$T[f(x)]=\sum_{k=0}^\infty \frac{p_k(x)}{k!} f^{(k)}(x)$ .
\end{problem}

A more general problem could also be taken from
Problem~\ref{differential_operator_problem} by removing the condition that
the linear transformation produce an OPS.

In Section 1, we gave a definition of orthogonal polynomials in terms of a
moment functional.  For the Hermite polynomials, the moment functional is
defined by
\[
\mathcal{L}[f(x)] = \int_{-\infty}^\infty f(x)e^{-x^2} dx.
\]
The moment functional for the general Laguerre polynomials is defined by
\[
\mathcal{L}[f(x)] = \int_0^\infty f(x)x^\alpha e^{-x} dx.
\]
The Jacobi polynomials $P_n^{\alpha,\beta}(x)$ for $\alpha, \beta \in \R$ are another type of OPS. The commonly known Chebyshev and Legendre polynomials are special cases of the these polynomials, and their moment functional is defined by
\[
\mathcal{L}[f(x)] = \int_{-1}^1 f(x) (1-x)^\alpha (1+x)^\beta dx.
\]
The above functionals can be found in Chihara~\cite{Chihara_Orthogonal_Polynomials}, p.~148.
Knowing that the last integral is defined on the interval $[-1,1]$, consider the following theorems.

\bigskip
\begin{theorem}
~\cite[Thm.~1, p.~559]{Iserles_and_Saff} Let the polynomial $\sum_{k=0}^n q_k x^k$ be a polynomial, with real coefficients $q_0,q_1,\ldots q_n$, have all of its zeros in the complex open unit disk. Then all of the zeros of $\sum_{k=0}^n q_kT_k(x)$, where $T_k(x)$ is the $kth$ Chebyshev polynomial of the first kind, lie in the open interval $(-1,1)$.
\end{theorem}
\noindent The same is true for the Chebyshev polynomials of the second kind
$U_n(x)$.

\bigskip
\begin{theorem}\cite[Thm.~1.2, p.~2]{Chasse}
If $f(x) = \sum_{k=0}^n a_k x^k$ has all of its zeros in the interval
$(-1,1)$, then $T[f(x)] = \sum_{k=0}^\infty a_k P_k(x)$ also has all of its
zeros in the interval $(-1,1)$, where $P_k(x)$ is the $kth$ Legendre
Polynomial.
\end{theorem}

\noindent With these theorems in place, we also present the following
problem.
\begin{problem}
Does the interval on which the moment functional for an OPS
$\{P_n(x)\}_{n=0}^{\infty}$ is defined relate to the real-root preserving
property of the transformation $T[x^n]=P_n(x)$ in a meaningful way?
\end{problem}

\section{Acknowledgments}
We would like to acknowledge Theodore Chihara for his book on orthogonal
polynomials. Studying the concepts in his book provided a strong framework
for us to consider representations of orthogonal polynomials in this paper.
In addition, we would like to acknowledge Andrzej Piotrowski for his work in
2007. Many of the ideas and concepts of this paper rely heavily on the ideas
given in his PhD thesis. We would also like to acknowledge Julius Borcea and
Petter Br\"and\'en for their highly insightful work in classifying all linear
real-root preservers. Their work has also contributed to the motivation of
this paper.


\bibliographystyle{amsplain}

\end{document}